\documentclass[11pt,letterpaper]{amsart}

\usepackage{tikz}
\tikzstyle{every node}=[circle, draw, fill=black!50,
                        inner sep=0pt, minimum width=3pt]

\usepackage{amsfonts,amsmath,latexsym,color,epsfig,hyperref,enumitem, amssymb,bbm}

\newtheorem{theorem}{Theorem}[section]
\newtheorem{lemma}{Lemma}[section]

\renewcommand{\le}{\leqslant}
\renewcommand{\ge}{\geqslant}

\def\qed{\ifvmode\mbox{ }\else\unskip\fi\hskip 1em plus 10fill$\Box$}

\makeatletter
\def\Ddots{\mathinner{\mkern1mu\raise\p@
\vbox{\kern7\p@\hbox{.}}\mkern2mu
\raise4\p@\hbox{.}\mkern2mu\raise7\p@\hbox{.}\mkern1mu}}
\makeatother

\def\R{\mathbb R}

\def\N{\mathbb N}
\def\E{\mathbb E}

\def\d{\delta}

\usepackage{bbm}   
\usepackage{scalerel,stackengine}
\stackMath
\newcommand\widecheck[1]{%
\savestack{\tmpbox}{\stretchto{%
  \scaleto{%
    \scalerel*[\widthof{\ensuremath{#1}}]{\kern-.6pt\bigwedge\kern-.6pt}%
    {\rule[-\textheight/2]{1ex}{\textheight}}
  }{\textheight}%
}{0.5ex}}%
\stackon[1pt]{#1}{\scalebox{-1}{\tmpbox}}%
}
\usepackage{ esint }

\title{A fractal-like configuration of point-line pairs for the minimal distance problem}

\author{Alexander Logunov, Dmitrii Zakharov}

\address{Department of Mathematics, Massachusetts Institute of Technology, Cambridge, MA 02139, USA}
\email{alogunov@mit.edu}

\address{Department of Mathematics, Massachusetts Institute of Technology, Cambridge, MA 02139, USA}
\email{zakhdm@mit.edu}

\date{}

\begin{document}

\begin{abstract}
    We show that for every $n \in \N$ there is a collection of points $p_1, \ldots, p_n$ and lines  $\ell_1, \ldots, \ell_n$ in the unit square such that for any $i$ we have $p_i \in \ell_i$ and the distance from $p_i$ to any other line $\ell_j$ is at least $c n^{\gamma-1}$ for some universal constants $c, \gamma>0$. This is better than a trivial construction by a polynomial factor.
\end{abstract}

\maketitle 

\section{Introduction}

In \cite{cohen2025lower}, Cohen, Pohoata and the second author considered the following problem about collections of point-line pairs. For an integer $n\ge 2$, what is the largest $\d>0$ such that there are points $p_1, \ldots, p_n \in [0,1]^2$ and lines $\ell_1, \ldots, \ell_n \subset \R^2$ such that $p_i \in \ell_i$ and for every $i\neq j$ we have $d(p_i, \ell_j) \ge \d$? By taking $p_i = (i/n, 0)$ and $\ell_i = p_i + (0,1)\R$ it is easy too see that $\d \ge 1/n$. On the other hand, by tiling $[0,1]^2$ with $\frac{1}{[\sqrt{n-1}]}\times \frac{1}{[\sqrt{n-1}]}$-squares we see that among any $n$ points in the unit square there is a pair at distance at most $\frac{\sqrt 2}{[\sqrt{n-1}]}$. So using that $d(p_i,\ell_j) \le d(p_i, p_j)$ this leads to the upper bound $\d \le \frac{\sqrt 2}{[\sqrt{n-1}]}$. 

In \cite{cohen2025lower}, authors managed to improve this simple upper bound to $\d \le n^{-2/3+o(1)}$. As a corollary, they obtained a new upper bound for Heilbronn's triangle problem which asks to determine the smallest number $\Delta(n)$ such that among any $n$ points in unit square there exists a triangle with area at most $\Delta(n)$. The connection between these two problems can be seen as follows: given a set of $n$ points in $[0,1]^2$ we can find a collection of at least $n/10$ disjoint pairs $\{p_i, q_i\}$ among these points so that $d(p_i, q_i) \le2 n^{-1/2}$ for all $i$. Let $\ell_i$ be the line passing through $p_i, q_i$ and note that if $d(p_i, \ell_j) \le \d$ for some $i\neq j$ then the points $p_i, p_j, q_j$ form a triangle of area at most $\Delta\le \frac 12 d(p_i, \ell_j) d(p_j, q_j) \le n^{-1/2}\d$. So better upper bounds on the minimal distance $\d$ lead to better upper bounds on the area of smallest triangle $\Delta$. In particular, the upper bound $\d\le n^{-2/3+o(1)}$ gives $\Delta \le n^{-7/6+o(1)}$.
This estimate improves on previous work on Heilbronn's triangle problem by Roth \cite{roth1972problem2}, \cite{roth1972problem3}, Koml\'os--Pintz--Szemer\'edi \cite{komlos1981heilbronn} and Cohen, Pohoata and the second author \cite{cohen2023new} and reaches a certain `high-low' limit of the method.  

Many authors (see e.g. the table on page 34 in \cite{cohen2023new}) have guessed that the bound $\Delta(n) \lesssim n^{-3/2}$ should hold based on the following heuristic.
Let $\ell_j = \overline{p_j q_j}$ be lines defined as above and let $H_j \subset [0,1]^2$ be the strip of points $x\in [0,1]^2$ at distance at most $\le K/n$ to $\ell_j$. If the points $p_1, \ldots, p_n$ were uniformly distributed in $[0,1]^2$ then we would expect that about $|H_j|n \sim K$ of them would land in the strip $H_j$. So for $K\gg 1$ there should be lots of points in the strip $H_j$ and so we should be able to find a triangle with area at most $2n^{-1/2} (K/n) \lesssim n^{-3/2}$. One way to turn this idea into a rigorous proof is to show that in the setup in the first paragraph, we have the upper bound $\d \lesssim n^{-1}$. Because of this, it seemed plausible to the authors of \cite{cohen2025lower} that the upper bound $\d\le n^{-1+o(1)}$ holds. In Appendix B of \cite{cohen2025lower}, they further note that this estimate, if true, would also imply non-trivial upper bounds for Heilbronn's problem for $k$-gons. 

In this note, we show that this is not the case, namely that there are collections of point-line pairs with $\d \gtrsim n^{\gamma -1}$ for some (small) constant $\gamma>0$. 

\begin{theorem}\label{thm:main}
    There exist constants $c,\gamma>0$ such that for any $n\ge 2$ there are points $p_1, \ldots, p_n \in [0,1]^2$ and lines $\ell_1, \ldots, \ell_n$ such that $p_i\in \ell_i$ and $d(p_i, \ell_j) \ge c n^{\gamma-1}$ for any $i\neq j$. 
\end{theorem}

Similar lower bounds hold for higher dimensional variants of this problem. For example, the same proof works for collections of point-line pairs in $[0,1]^d$ for any $d\ge 2$, we omit the details since they are essentially identical. On the other hand, Theorem \ref{thm:main} does not imply anything new about Heilbronn's triangle problem. However it does hint at the possibility that there could be better lower bound constructions: the current best lower bound $\Delta(n) \gtrsim \frac{\log n}{n^2}$ due to Koml\'os--Pintz--Szemer\'edi \cite{komlos1982lower} is a slight improvement over the `easy' lower bound $\Delta(n) \gtrsim n^{-2}$ due to Erd\H os.

In the next section we present the proof of Theorem \ref{thm:main}. The proof consists of two parts: first, we use the probabilistic method to beat the lower bound by a logarithmic factor and second, we use a recursive construction to amplify this improvement into a polynomial gain. In particular, the set of points and lines we construct is essentially a self-affine fractal. Self-affine fractals might be an interesting special class of configurations for further study of the minimal distance problem.

\section{Proof}

To ease the recursive construction it is convenient to introduce the following notation (very similar to the one used in \cite{cohen2025lower}). 
Let $\Omega = [-1,1]^2\times [-1,1]$, where we view $\Omega$ as the configuration space of incident point-line pairs in the plane. Namely, for $(p, \theta) \in \Omega$ let $\ell =\ell(p, \theta) = p + (1,\theta) \R$ be the line with slope $\theta$ through $p$. So we can identify a subset $X \subset \Omega$ with the point-line configuration $\{(p, \ell(p, \theta)), ~~(p, \theta)\in X\}$.

For $\omega=(x,y, \theta), \omega'=(x', y', \theta') \in \Omega$ define the vertical distance $d(\omega, \omega') = |y - y' -\theta'(x-x') |$. For $\omega, \omega'\in \Omega$, $d(\omega, \omega')$ measures the vertical distance from the point $p=(x, y)$ to the line $\ell' = \ell'(p', \theta')$, with $p'=(x', y')$. Since $|\theta'| \le 1$ it coincides with the euclidean distance up to a constant factor. 
For a point-line configuration $X \subset \Omega$ define the {\em minimal distance} of $X$ as
\[
d(X) = \min_{\omega\neq \omega'\in X} d(\omega, \omega').
\]

For $\d>0$ define $n(\d)$ to be the largest integer $n$ such that there exists $X \subset \Omega$ with $d(X) \ge \d$ and $|X| = n$. With this notation, it is easy to see that Theorem \ref{thm:main} is equivalent to the following.

\begin{theorem}\label{thm}
    There exist constants $\gamma,c>0$ such that $n(\d) \ge c\d^{-1-\gamma}$ for all $\d\in (0,1)$. 
\end{theorem}

Theorem \ref{thm} now follows the following two lemmas.

\begin{lemma}\label{lem1}
    There exists a constant $c_1>0$ such that $n(\d) \ge c_1 \d^{-1} \log 1/\d$ for all $\d\in (0,1)$.
\end{lemma}

\begin{lemma}\label{lem2}
    There exists a constant $C_2\ge 1$ such that for any $w, \d \in (0,1)$ such that $\d \le w^2$ we have $n(\d) \ge w^{-1} n(C_2w)n(\d/w^2)$.
\end{lemma}

\begin{proof}[Proof of Theorem \ref{thm}]
    Let $w\in (0,1)$ be such that $\log (1/C_2w)\ge e c_1^{-1} C_2$ holds. Then for any $\d \le w^{2}$ we have
    \[
    n(\d) \ge w^{-1} n(C_2 w) n(\d/w^2) \ge c_1 C_2^{-1}w^{-2} \log(1/C_2w)n(\d/w^2) \ge e w^{-2} n(\d/w^2).
    \]
    Now fix $w$ and let $\d \in [w^{2m+2}, w^{2m}]$ for an integer $m\ge 1$, then iterating this inequality $m$ times gives
    \[
    n(\d) \ge e^m w^{-2m} n(\d/w^{2m}) \ge e^m w^{-2m} \ge e^m w^2 \d^{-1} \ge c \d^{-1-\gamma}
    \]
    for some $c=c(w)>0$ and $\gamma=\gamma(w)>0$.
\end{proof}

In the remainder of the section we prove the two lemmas.

\begin{proof}[Proof of Lemma \ref{lem1}]
    If is enough to prove the lemma for all sufficiently small $\d \le\d_0$ (for $\d>\d_0$ we can just choose $c_1$ small enough). For $p\in \R^2$ and $\theta \in [-1, 1]$ let $T(p, \theta) = p + (1, \theta)\R + \{0\} \times [-\d, \d]$, i.e. $T(p, \theta)$ is a strip of vertical width $2\d$ through $p$ at slope $\theta$.
    
    Let $N = [0.1 \d^{-1} (\log 1/\d)]$. 
    Let $p_1, \ldots, p_N \in [0,1]^2$ be a uniformly random and mutually independent collection of points and let $P = \{p_1, \ldots, p_N\}$. 
    For each point $p_i$, it suffices to show that the probability that there exists an angle $\theta\in [-1,1]$ such that $T(p_i, \theta) \cap P = \{p_i\}$ is at least $1/2$ (for $\d\le \d_0$). Indeed, by linearity of expectation this would give a collection of $N/2$ points among $p_1, \ldots, p_N$ so that for every $p_i$ in the collection there is an angle $\theta$ such that $T(p_i, \theta)\cap P =\{p_i\}$. In particular, this gives $n(\d) \ge N/2 \ge c \d^{-1}\log 1/\d$.
    
    Fix an index $i$. 
    Let $Q_i = p_i + [- \frac{1}{4\sqrt{N}}, \frac{1}{4\sqrt{N}}]^2$. 
    Note that for $i'\neq i$ we have $p_{i'}\in Q_i$ with probability at most $|Q_i|\le \frac{1}{4N}$. So by the union bound we have
    \begin{equation*}\label{eq:Q}
    \Pr[Q_i \cap P \neq \{p_i\}] \le (N-1) |Q_i| \le 1/4.    
    \end{equation*}
    So with probability at least $3/4$ there are no points inside $Q_i$ besides $p_i$. So it suffices to show that with probability at least $3/4$ there exists a slope $\theta$ such that $(T(p_i, \theta) \setminus Q_i)\cap P = \emptyset$. 

    Let $P' = P\setminus\{p_i\}$.
    Denote $M = [\frac{1}{8\d\sqrt{N}}]$ and let $\theta_j = 8j \d \sqrt{N}$, for $j=1, \ldots, M$. Note that for any $j\neq j' \in \{1, \ldots, M\}$ we have $T(p_i, \theta_j) \cap T(p_i, \theta_{j'})\subset  Q_i$. We now estimate the probability that $P' \cap (T(p_i, \theta_j) \setminus Q_i) = \emptyset$ for some $j \in \{1, \ldots, M\}$. 
    For a region $A \subset \R^2$ the probability that $A\cap P' = \emptyset$ is equal to $(1 - |A \cap [0,1]^2|)^{N-1}$. By mutual independence of the collection $\{p_1 \ldots, p_N\}$, this also holds if $A$ is a random set depending only on $p_i$. Since $T(p_i,\theta)$ has vertical width $2\d$, we have $|T(p_i,\theta) \cap [0,1]^2| \le 2\d$. So we can choose pairwise disjoint sets $A_j\subset [0,1]^2$ of area $2\d$ so that $A_j$ contains $T(p_i, \theta_j) \cap [0,1]^2 \setminus Q_i$.

    
    Let $Z_j = 1_{A_j \cap P' = \emptyset}$ be the indicator of the event that $A_j \cap P' = \emptyset$. Then we have 
    \[
    \E Z_j = (1-2\d)^{N-1} \ge \exp(- 2\d (N-1)) \ge \exp(-2\d \times 0.1\d^{-1} \log (1/\d)) \ge \d^{0.2}.
    \]
    Denote $q = (1-2\d)^{N-1}$. So the sum $Z= \sum_{j=1}^M Z_j$ has expectation $\E Z =M q$. Since $M = [\frac{1}{8\d\sqrt{N}}] \sim (\d \log(1/\d))^{-1/2} \gg \d^{-0.2}$, we have $Mq > 16$ for $\d \le \d_0$. We note that the variables $Z_j$ are negatively correlated: for $j\neq j'$ the product $Z_{j} Z_{j'}$ is the indicator of the event that $P' \cap (A_{j}\cup A_{j'}) = \emptyset$ and so $\E Z_j Z_{j'} = (1-4\d)^{N-1} \le q^2 = (\E Z_j) (\E Z_{j'})$. For $j=j'$ we have $\E Z_jZ_{j'} = \E Z_j = q$. 
    We use this to estimate the variance of $Z$:
    \begin{align*}
    \operatorname{Var}Z = \E Z^2 - (\E Z)^2= \sum_{j, j'=1}^M \left(\E Z_j Z_{j'}- (\E Z_j)(\E Z_{j'}) \right)\le  M q.
    \end{align*}
    By Chebyshev's inequality,
    \[
    \Pr[ |Z -\E Z| \ge \E Z/2] \le 4\frac{\operatorname{Var}Z}{(\E Z)^2} \le \frac{4}{M q} \le 1/4.
    \]
    We conclude that for every $i$, the probability that there exists $\theta$ with $T(p_i, \theta)\cap P=\{p_i\}$ is at least $1/2$ which in turn implies that $n(\d) \ge N/2 \ge c\d^{-1}\log 1/\d$, as desired.
\end{proof}

Fix $w>0$. For $(p_0, \theta_0)=(x_0,y_0, \theta_0)\in \Omega$ let us define an affine map $\psi^{(w)}_{p_0, \theta_0}: \R^2\to \R^2$ given by
\[
\psi^{(w)}_{p_0, \theta_0}(x, y) = (x_0 + w x, y_0 + w\theta_0 x + w^2 y)
\]
and extend it to a map on $\R^2\times \R$ by
\[
\psi^{(w)}_{p_0, \theta_0}(x, y, \theta) = (x_0 + w x, y_0 + w\theta_0 x + w^2 y, \theta_0 + w \theta).
\]
We note that if $|x_0|, |y_0|, |\theta_0| \le 1/2$ and $w\le 1/2$, then $\psi^{(w)}_{p_0, \theta_0}(\Omega) \subset \Omega$, i.e. we get a well-defined map on the configuration space of point-line pairs. Roughly speaking, the map $\psi^{(w)}_{p_0, \theta_0}$ sends a point line pair inside $[-1,1]^2$ to a point-line pair inside of a $w\times w^2$ rectangle centered at $p_0$ and inclined by slope $\theta_0$. 

The rescaling map $\psi=\psi^{(w)}_{p_0,\theta_0}$ interacts well with the distance $d$: a straightforward computation shows that for $\omega = (x,y, \theta), \omega'=(x',y',\theta')$ we have
\[
d(\psi(\omega), \psi(\omega')) = w^2 |y - y' - \theta'(x-x')| = w^2 d(\omega,\omega').
\]
In particular, for any $X \subset \Omega$ we have $d(\psi(X)) = w^2 d(X)$. Now we use the rescaling maps to prove Lemma \ref{lem2}.

\begin{proof}[Proof of Lemma \ref{lem2}]
    Fix a sufficiently large constant $C$ and let $w \in (0, 1/4C^2)$. Let $X_0 \subset \Omega$ be a configuration of size $k=n(4C^2w)$ such that $d(X_0)\ge 4C^2w$. Then $X_1 = \psi^{(1/2)}_{(0,0),0}(X_0) \subset [-1/2,1/2]\times[-1/4,1/4]\times [-1/2, 1/2]$ is a configuration with $d(X_1) \ge C^2 w$.
    Write $X_1 = \{ (p_1, \theta_1), \ldots, (p_k, \theta_k) \}$. 
    For $|j| \le w^{-1}$ let $p_{i, j} = p_i + (0, Cj w^2)$. Note that if $C\ge 4$ then $p_{i, j} \in [-1/2, 1/2]^2$. 
    For each $i=1, \ldots,k$ and $|j| \le w^{-1}$ we then have an affine map $\psi_{i,j} = \psi_{p_{i,j}, \theta_i}^{(w)}: \Omega\to \Omega$.

    Now given an arbitrary $X \subset \Omega$ we define a new configuration $\tilde X$ by
    \[
    \tilde X = \bigcup_{i\in [k], |j| \le w^{-1}} \psi_{i,j}(X).
    \]
    We claim that this is a disjoint union and for an appropriate $C$ we have $d(\tilde X) \ge w^2 \min(d(X),1)$. Given this, we can finish the proof as follows: by construction, we have $|\tilde X| = k (2[w^{-1}]-1) |X| \ge n(4C^2 w) w^{-1}|X|$. For $\d \le w^2$ and $w \le 1/4C^2$ we then get
    \[
    n(\d) \ge |\tilde X| \ge n(4C^2 w) w^{-1}|X| \ge  n(4C^2 w) w^{-1} n(\d/w^2).
    \]
    
    Now it remains to show that for any $(\tilde p, \tilde \theta) \neq (\tilde p', \tilde \theta') \in \tilde X$ we have $d((\tilde p, \tilde \theta), (\tilde p', \tilde \theta')) \ge w^2 \min(d(X),1)$. We can write
    \[
    \tilde \omega = \psi_{i,j}(p, \theta),\quad \tilde \omega' = \psi_{i', j'}(p', \theta'),
    \]
    for some $i, i' \in [k]$, $|j|, |j'|\le w^{-1}$ and $(p,\theta), (p', \theta') \in X$. We split into three cases:
    \begin{itemize}
        \item Suppose that $i\neq i'$. One can check that we have $\tilde p=\psi_{i,j}(p) \in p_i + [-w, w]\times [-2Cw, 2C w]$ and $\tilde p'=\psi_{i', j'}(p') \in p_{i'}+ [-w, w]\times [-2Cw, 2C w]$. Furthermore, we have $\tilde \theta' \in \theta_{i'} + [-w,w]$ and so we get
        \begin{align*}
        d((\tilde p, \tilde \theta), (\tilde p', \tilde \theta')) &\ge d((p_i, \theta_i), (p_{i'}, \theta_{i'})) - (4C+4)w \\&\ge d(X_1) -(4C+4)w \ge (C^2-4C-4) w \ge w    
        \end{align*}
        for $C\ge 5$.

        \item Suppose that $i=i'$ but $j\neq j'$. In this case, one can check that for $p = (x, y)$ and $p'=(x', y')$
        \[
        d((\tilde p, \tilde \theta), (\tilde p', \tilde \theta')) = |Cw^2(j-j') +  w^2 (y - y' - \theta'(x-x'))| \ge Cw^2 - 4w^2 \ge w^2,
        \]
        where we used $|j-j'| \ge 1$ and $C\ge 5$. 
        \item Finally, suppose that $i=i'$, $j=j'$ and $(p, \theta) \neq (p', \theta')$. Then we know that 
        \[
        d((\tilde p, \tilde \theta), (\tilde p', \tilde \theta')) = w^2 d((p, \theta), (p', \theta')) \ge w^2 d(X).
        \]
    \end{itemize}

    Putting the three cases together this implies the desired bound on $d(\tilde X)$ which completes the proof.
\end{proof}

\bibliographystyle{amsplain0.bst}
\bibliography{main}
\end{document}